\documentclass{amsart}

\usepackage{amssymb}
\usepackage{color}

\newtheorem{theorem}{Theorem}[section]
\newtheorem{lemma}[theorem]{Lemma}

\theoremstyle{definition}
\newtheorem{definition}[theorem]{Definition}

\newtheorem{proposition}[theorem]{Proposition}

\theoremstyle{remark}
\newtheorem{remark}[theorem]{Remark}

\numberwithin{equation}{section}

\def\EE{{\mathcal{E}}}
\def\FF{{\mathcal{F}}}

\begin{document}

\title[Criticality of Schr\"odinger forms]{Criticality and subcriticality of generalized Schr\"odinger forms with non-local perturbations}


\author[L. Li]{Liping Li}
\address{Institute of Applied Mathematics, Academy of Mathematics and Systems Science, Chinese Academy of Sciences, Beijing 100190, China.}
\curraddr{}
\email{liping\_li@amss.ac.cn}
\thanks{{Partially supported by a joint grant (No. 2015LH0043) of China Postdoctoral Science Foundation and Chinese Academy of Science, and China Postdoctoral Science Foundation (No. 2016M590145).}}


\date{}

\dedicatory{}

\commby{David Asher Levin}

\begin{abstract}
In this paper, we shall treat the Schr\"odinger forms with non-local perturbations. We first extend the definitions of subcriticality, criticality and supercriticality for the Schr\"odinger forms by Takeda \cite{T14-2} to the non-local cases in the context of quasi-regular Dirichlet forms. Then we prove an analytic characterization of these definitions via the bottom of the spectrum set. 
\end{abstract}

\maketitle

\section{Introduction}\label{SEC1}

A new method to define the subcriticality, criticality and supercriticality of the Schr\"odinger form was introduced by Takeda (see \cite[Definition~4.4]{T14-2}). These definitions are given by the existence of superharmonic functions and the global property (recurrence or transience) of Dirichlet forms. Furthermore, an analytic characterization of these definitions is described as follows: let $\lambda(\mu)$ be  the bottom of the spectrum of the generalized Schr\"odinger operator after a time-change transform (see \cite[(5.1)]{T14-2}), then the subcriticality, criticality and supercriticality are equivalent to $\lambda(\mu)>1$, $\lambda(\mu)=1$ and $\lambda(\mu)<1$ respectively. In this paper, we shall extend these results to the Schr\"odinger forms with non-local perturbations. 

Let $E$ be a Lusin space and $m$ a fully supported $\sigma$-finite measure on $E$. The Markov process $X$ is $m$-symmetric on $E$ associated with a quasi-regular Dirichlet form $(\EE,\FF)$ on $L^2(E,m)$. Basically, we enforce the following assumptions on $X$: 
\begin{description}
\item[(I)] $X$ is irreducible and transient;
\item[(SF)] the semigroup $(P_t)$ of $X$ is strong Feller, i.e. $P_t(\mathcal{B}_b(E))\subset C_b(E)$ for each $t>0$. 
\end{description}
Here $\mathcal{B}_b(E)$ and $C_b(E)$ stand for the families of all bounded Borel functions and all bounded continuous functions on $E$. Clearly, \textbf{(SF)} implies the following condition: 
\begin{description}
\item[(AC)] $(P_t)$ is absolutely continuous with respect to $m$, i.e. there exists a density function $p(t,x,y)$ such that ${P_t}(x,dy)=p(t,x,y)m(dy)$ for each $t>0$. 
\end{description}
Under \textbf{(AC)}, $X$ has a non-negative, jointly measurable $\alpha$-resolvent kernel $G_\alpha(x,y)$ for $\alpha\geq 0$. Note that $G_\alpha(x,y)$ is $\alpha$-excessive both in $x$ and $y$. We simply write $G(x,y)$ for $G_0(x,y)$. We further denote the L\'evy system, jumping measure, generator and lifetime of $X$ by $(N, H)$, $J$, $\mathcal{L}$ and $\zeta$. Note that {$J(dxdy)=(1/2) N(x,dy)\mu_H(dx)$}, where $\mu_H$ is the associated Revuz measure of $H$. Moreover, $\EE(u,v)=(-\mathcal{L}u,v)_m$ for $u\in \mathcal{D}(\mathcal{L})$ and $v\in \FF$.  All the terminologies above are standard, and we refer them to \cite{CF12, FOT11, MR92}. 

Consider an additive functional  (AF in abbreviation) $A$ as follows:
\begin{equation}\label{EQ1ATA}
	A_t:= A_t^\mu + \sum_{s\leq t} F(X_{s-},X_s), \quad t\geq 0, 
\end{equation}
where $A^\mu$ is the associated continuous additive functional (CAF in abbreviation) of a signed smooth measure $\mu$, and $F$ is a bounded  symmetric function (i.e. $F(x,y)=F(y,x)$) on $E\times E\setminus d$ ($d$ is the diagonal of $E\times E$). Write the natural decompositions
\[
	\mu=\mu^+- \mu^-, \quad A^\mu=A^{\mu^+}- A^{\mu^-},\quad F=F^+-F^-, 
\]
where $\mu^+,\mu^-$ are positive smooth measures, $A^{\mu^+}, A^{\mu^-}$ are their associated positive continuous additive functionals (PCAF in abbreviation), and $F^+:=F\vee 0, F^-:=-(F\wedge 0)$ are positive symmetric functions. Further denote the following AFs:
\[
A^{F^\pm}_t:= \sum_{s\leq t} F^\pm(X_{s-},X_s),\quad A^\pm_t:= A^{\mu^\pm}_t+A^{F^\pm}_t. 
\]
Clearly, $A=A^+-A^-$. 
Let
\begin{equation}\label{EQ1EMF}
	\mathrm{e}_{\mu+F}(t):=\exp(A_t).
\end{equation}
 The Feynman-Kac semigroup after the perturbation by $A$ is defined by
\begin{equation}\label{EQ1PAT}
	P^{-A}_tf(x):= \mathbf{E}_x\left( f(X_t)\mathrm{e}_{\mu+F}(t)\right),\quad t\geq 0, \ f\in \mathcal{B}_b(E). 
\end{equation}
Under appropriate conditions (see Lemma~\ref{LM21}), $(P^{-A}_t)$ is a strongly continuous semigroup on $L^2(E,m)$ {and its} associated quadratic form $(\EE^{-A},\FF^{-A})$ is a lower bounded, closed and symmetric form on $L^2(E,m)$. This quadratic form is also called a Schr\"odinger form in \cite{T14-2}. Usually, $(P_t^{-A})$ is not Markovian. Following \cite{T14-2} with some modifications,  we set a class of positive superharmonic functions of $(P_t^{-A})$ as follows:
\[
	\mathcal{H}^+:=\left\{h: 0<h<\infty, \text{ q.e.}, h\text{ is quasi-continuous and }P^{-A}_th\leq h \text{ for each }t \right\}, 
\]
where the quasi-notions are relative to $(\EE,\FF)$. If $h\in \mathcal{H}^+$, then the $h$-transformed semigroup
\[
	P^{-A,h}_tf(x):=\frac{1}{h(x)}P_t^{-A}(fh)(x)
\]
is naturally Markovian, and its associated quadratic form
\begin{equation}\label{EQ1FAH}
\begin{aligned}
 \FF^{-A,h}&:=\{u: u\cdot h\in \FF^{-A}\},\\
 \EE^{-A,h}(u,v)&:=\EE^{-A}(uh, vh),\quad u,v\in \FF^{-A,h}, 
\end{aligned}
\end{equation}
is a Dirichlet form on $L^2(E,h^2\cdot m)$. The following definitions are taken from \cite[Definition~4.4]{T14-2}. We shall prove in Proposition~\ref{PRO24} that they are still well-defined for the cases of non-local perturbations. 

\begin{definition}\label{DEF11}
The Schr\"odinger form $(\EE^{-A},\FF^{-A})$ is said to be 
\begin{itemize}
\item[(1)] \emph{subcritical} if $(\EE^{-A,h},\FF^{-A,h})$ is transient for some $h\in \mathcal{H}^+$;
\item[(2)] \emph{critical} if $(\EE^{-A,h},\FF^{-A,h})$ is recurrent for some $h\in \mathcal{H}^+$;
\item[(3)] \emph{supercritical} if $\mathcal{H}^+=\emptyset$. 
\end{itemize}
\end{definition}

{Let us state the main result of this paper. Set 
\[
	G^-:=1-\mathrm{e}^{-F^-},\quad G^+:=\left(\mathrm{e}^{F^+}-1\right)\cdot \mathrm{e}^{-F^-}
\]
and write $G:=G^+-G^-$.} Since $F$ is bounded, there exists a constant $C>0$ such that
\[
C^{-1}F^+(x,y)\leq G^+(x,y)\leq CF^+(x,y),\quad C^{-1}F^-(x,y)\leq G^-(x,y)\leq  CF^-(x,y)
\]
for any $x,y\in E\times E\setminus d$. 
{Further denote
\begin{equation}\label{EQ1RMN}
	\rho^+:= \mu^++ NG^+\cdot \mu_H, \quad \rho^-:= \mu^-+ NG^-\cdot \mu_H,\quad \rho:= \rho^+-\rho^-. 
\end{equation}
Then $(\EE^{-A},\FF^{-A})$ is identified in Lemma~\ref{LM21}. }
The bottom of the spectrum of Schr\"odinger form $(\EE^{-A}, \FF^{-A})$ (Cf. \cite[(5.1)]{T14-2}) is now replaced by 
\begin{equation}
\lambda(\mu+F):= \inf\left\{\EE^{-A}(u,u): u\in \FF^{-A}, \ \int_E u^2d\rho^+=1 \right\}. 
\end{equation}
Under some appropriate conditions, the subcriticality, criticality and supercriticality of $(\EE^{-A},\FF^{-A})$ are equivalent to $\lambda(\mu+F)>0$,  $\lambda(\mu+F)=0$ and $\lambda(\mu+F)<0$ respectively (Theorem~\ref{THM31}). Particularly, when $F\equiv 0$, $\EE^{-A}(u,v)=\EE(u,u)-\int_E u^2d\mu$ and $\rho^+=\mu^+$. Thus this is an extended result of \cite[Theorem~5.19]{T14-2}. 

\section{Subclasses of smooth measures}

Several subclasses of smooth measures and bivariate functions will appear in this paper. For the readers' convenience,  we make a brief summary here. 

Fix a right Markov process $X$. 
The classes $\mathbf{K}(X), \mathbf{K}_1(X)$ and $\mathbf{K}_\infty(X)$ are some Kato-type classes of smooth measures and defined in \cite[Definition~2.1 and 2.2]{C02}. Note that under \textbf{(SF)}, $\mathbf{K}_\infty(X)$ equals the class of Green-tight measures $\mathcal{K}_\infty$ in \cite[Definition~2.2]{T14-2} (see \cite[Lemma~4.1]{KK15}).  The classes $\mathbf{J}(X)$ and $\mathbf{J}_\infty(X)$ defined in \cite[Definition~2.4]{C02} are the counterparts of $\mathbf{K}(X)$ and $\mathbf{K}_\infty(X)$ for the bivariate functions on $E\times E\setminus d$. Precisely, a function $F$ is said to be in $\mathbf{J}(X)$ (resp. $\mathbf{J}_\infty(X)$) if $F$ is bounded and $N|F|\cdot \mu_H\in \mathbf{K}(X)$ (resp. $\mathbf{K}_\infty(X)$).  Another subclass of bivariate functions $\mathbf{A}_\infty(X)$ defined in \cite[Definition~3.4]{C02} will be used in Theorem~\ref{THM31}. 

We also take two notations from \cite{KK15}. One is the Dynkin class $\mathbf{S}^1_D(X)$, and another one is the extended Kato class {$\mathbf{S}^1_{EK}(X)$} (Cf. \cite[\S2]{KK15}). Let $\mathbf{S}_1(X)$ be the family of smooth measures in the strict sense (Cf. \cite{FOT11}). A measure $\nu\in \mathbf{S}_1(X)$ is said to be in $\mathbf{S}_D^1(X)$ if $\sup_{x\in E}G_\alpha|\nu|(x)< \infty$ for some $\alpha>0$ and in $\mathbf{S}^1_{EK}(X)$ if $\lim_{\alpha\rightarrow \infty}\sup_{x\in E} G_\alpha |\nu|(x)<1$. 

The following facts can be deduced easily or found in the given reference:
\begin{itemize}
\item[(1)] $\mathbf{K}(X), \mathbf{K}_\infty(X), \mathbf{J}(X), \mathbf{J}_\infty(X)$ and $\mathbf{S}^1_D(X)$ are all linear spaces;
\item[(2)] $\mathbf{K}_\infty(X)\subset \mathbf{K}(X) \cap \mathbf{K}_1(X)$ (Cf. \cite[Proposition~2.3]{C02});
\item[(3)] $\mathbf{A}_\infty(X)\subset \mathbf{J}_\infty(X)\subset \mathbf{J}(X)$,  (Cf. \cite[\S3.2]{C02}); 
\item[(4)] $\mathbf{K}(X)\subset \mathbf{S}^1_{EK}(X)\subset \mathbf{S}^1_D(X)$;
\item[(5)] if $Y$ is another symmetric Markov process with the $0$-order resolvent $G^Y$ such that the smooth measure of $X$ is also that of $Y$ and $G^Y$ is bounded by $K\cdot G$ with some constant $K>0$ (for example, $Y$ is the subprocess of $X$ by killing, see \cite{Y96}), then $\mathbf{K}(X)\subset \mathbf{K}(Y)$, $\mathbf{K}_\infty(X)\subset \mathbf{K}_\infty(Y)$ and $\mathbf{S}^1_D(X)\subset \mathbf{S}_D^1(Y)$.
\end{itemize}
Moreover, let $\nu$ be a positive smooth measure. For any $u\in \FF$ and $\beta \geq 0$, it holds 
\begin{equation}\label{EQ2EUN}
	\int_E u^2d\nu \leq \|G_\beta \nu\|_\infty \EE_\beta(u,u). 
\end{equation}
This fact is taken from \cite[Proposition~4.2]{CS03} and \cite{SV96}. We shall use it frequently. 

\section{Criticality and subcriticality}

{The perturbation by \eqref{EQ1ATA} may be decomposed into two steps: the first step is killing by $\mathrm{e}^{-A^-_t}$, and the second step is the perturbation by the positive part $A^+$ of \eqref{EQ1ATA}. We give more explanations about these two procedures in the following remark.

\begin{remark}\label{RM31}
We first note that $t\mapsto \mathrm{e}^{-A^-_t}$ is a decreasing multiplicative functional of $X$, which never vanishes before $\zeta$. Thus this killing transform gives the semigroup $P_t^{A^-}f=\mathbf{E}_x\left(f(X_t)\mathrm{e}^{-A^-_t}\right)$. It is still symmetric with respect to $m$ and the associated Dirichlet form is given by (Cf. \cite{Y96})
\begin{equation}\label{EQ3FAF}
\begin{aligned}
	\FF^{A^-}&=\FF\cap L^2(E,\rho^-),\\
	\EE^{A^-}(u,u)&=\EE(u,u)-\int_{E\times E\setminus d}(u(x)-u(y))^2 G^-(x,y)J(dxdy)\\ &\qquad \qquad \qquad \qquad \qquad +\int u^2d\rho^-,\qquad\qquad u\in \FF^{A^-}. 
\end{aligned}
\end{equation}
Denote the subprocess of $X$ after killing by $X^{A^-}$. We realize $X^{A^-}$ in the same sample path space as $X$ and attain a new class of probability measures $(\mathbf{P}^{A^-}_x)_{x\in E}$.  Then the L\'evy system of $X^{A^-}$ is $(N^-,H)$, where 
\begin{equation}\label{EQ3NXY}
	N^-(x,dy)=\left(1-G^-(x,y)\right)N(x,dy)=\mathrm{e}^{-F^-(x,y)}N(x,dy). 
\end{equation}
Particularly, the jumping measure of $X^{A^-}$ is $J^-(dxdy)=(1-G^-(x,y))J(dxdy)$. 

The second step is the perturbation  of $X^{A^-}$ induced by $A^+$. In fact, we have  (Cf. \cite[(62.19)]{S88})
\[
	\mathbf{E}^{A^-}_x(f(X_t)\mathrm{e}^{A^+_t})= \mathbf{E}_x(f(X_t)\mathrm{e}^{A^+_t}\mathrm{e}^{-A^-_t})= P_t^{-A}f(x). 
\]
Mimicking \cite{Y97} (see also \cite{C03, CS03, KK15}), this perturbation can be also decomposed into two parts: one is the transform induced by the supermartingale multiplicative functional 
\begin{equation}\label{EQ3LTA}
	L_t:=A^{F^+}_t- \int_0^tN^-\left(\mathrm{e}^{F^+}-1\right)(X_s)dH_s=A^{F^+}_t- \int_0^t NG^+(X_s)dH_s,
\end{equation}
and another one is the perturbation by the positive continuous additive functional 
\[
	A^{+*}_t:=A^{\mu^+}_t+  \int_0^tN^-\left(\mathrm{e}^{F^+}-1\right)(X_s)dH_s.
\]
The Revuz measure of $A^{+*}$ relative to $X^{A^-}$ is $\rho^+$, which can be easily obtained by using \cite[Theorem~4.2]{Y98}. 
\end{remark}

We say $(\EE^{-A},\FF^{-A})$ is lower bounded if there exists a constant $\alpha$ such that for any $u\in \FF^{-A}$, $\EE^{-A}_\alpha(u,u)\geq 0$. 
}

\begin{lemma}\label{LM21}
Assume that $\rho^-\in \mathbf{S}^1_D(X)$ and $\rho^+\in \mathbf{S}^1_{EK}(X^{A^-})$. Then the Feynman-Kac semigroup \eqref{EQ1PAT} is strongly continuous on $L^2(E,m)$. Furthermore, its associated quadratic form is
\begin{equation}\label{EQ2FAF}
\begin{aligned}
	\FF^{-A}&=\FF,\\
	\EE^{-A}(u,v)&= \EE(u,v)-\int_E u(x)v(x)\mu(dx) \\ 
			 &\qquad \qquad \quad -\int_E\int_E u(x)v(y)G(x,y)N(x,dy)\mu_H(dx),
\end{aligned}
\end{equation}
which is a lower bounded positivity preserving symmetric form  (\cite{MR95}) on $L^2(E,m)$. 
\end{lemma}
{
\begin{proof}
Let $(\EE^{A^-}, \FF^{A^-})$ be given by \eqref{EQ3FAF}.
Since $F$ is bounded and $\rho^-\in \mathbf{S}^1_D(X)$, it follows from \eqref{EQ2EUN} that $\FF^{A^-}=\FF$ and $\EE^{A^-}_1\asymp \EE_1$. Here and hereafter, $\EE^{A^-}_1\asymp \EE_1$ means there exists a constant $C>0$ such that for any $u\in \FF$, 
\[
	C^{-1}\EE_1(u,u)\leq \EE^{A^-}_1(u,u)\leq C\EE_1(u,u). 
\]
Applying Remark~\ref{RM31} and \cite{Y97}, we may conclude \eqref{EQ1PAT} is strongly continuous and its associated quadratic form is 
\[
\begin{aligned}
	\FF^{-A}&=\FF, \\
	\EE^{-A}(u,u)&=\EE^{A^-}(u,u)+\int (u(x)-u(y))^2\left(\mathrm{e}^{F^+}-1\right)(x,y)J^-(dxdy)\\&\qquad \qquad \qquad -\int u^2d\rho^+,\qquad\qquad u\in \FF^{-A}. 
\end{aligned}	
\]
It follows from \eqref{EQ3FAF} and \eqref{EQ3NXY} that for any $u\in \FF$, 
\[
	\EE^{-A}(u,u)= \EE(u,u)+\int (u(x)-u(y))^2G(x,y)J(dxdy)-\int u^2d\rho. 
\]
Clearly, the right side of above is equal to the right side of the second equality in \eqref{EQ2FAF}. 
Similarly, from $\rho^+\in \mathbf{S}^1_{EK}(X^{A^-})$ we may deduce that {$\mathcal{E}^{-A}_1\asymp \EE^{A^-}_\alpha$ for some large $\alpha>0$. Thus $\mathcal{E}^{-A}_\alpha \asymp \EE_1$ for such $\alpha>0$.} That indicates $(\EE^{-A},\FF^{-A})$ is a lower bounded, symmetric closed form on $L^2(E,m)$. Then the left conclusions follow from \cite[Theorem~6.2.1]{Y97} and \cite[Theorem~1.5]{MR95}. 
\end{proof}}

\begin{remark}
 A different exponent (say, the Stieltjes exponent) was treated in \cite{Y96}. {Without loss of generality, assume $F^-=0$ and $\mu^-=0$. Consider the right continuous increasing AF 
 \[
 	A^{\mu^++G^+}:=A_t^{\mu^+}+\sum_{s\leq t}G^+(X_{s-},X_s).
 \]
 Its Stieltjes exponent is denoted by}
 \[
 	(\mathrm{Exp}A^{\mu^++G^+})_t:= \mathrm{e}^{A^{\mu^++G^+, c}_t}\prod_{s\leq t}(1+\Delta A^{\mu^++G^+}_s),
 \] 
where $A^{\mu^++G^+, c}$ is the continuous part of $A^{\mu^++G^+}$ and $\Delta A^{\mu^++G^+}_s:= A^{\mu^++G^+}_s-A^{\mu^++G^+}_{s-}$.
 Clearly, we have
\[
	(\mathrm{Exp}A^{\mu^++G^+})_t= \mathrm{e}^{A^{\mu^+}_t} \prod_{s\leq t}(1+G^+(X_{s-}, X_s))=\mathrm{e}^{A^{\mu^+}_t+\sum_{s\leq t}F^+(X_{s-},X_s)}=\mathrm{e}^{A^+_t}. 
\]
This is the significance we introduce the function $G$ in \S\ref{SEC1}.  
\end{remark}

Now we assert that if $\mathcal{H}^+$ is not empty, then $(\EE^{-A,h},\FF^{-A,h})$ is associated with a right Markov process for any $h\in \mathcal{H}^+$. 

\begin{proposition}\label{PRO23}
Assume that {$\rho^-\in \mathbf{S}^1_D(X), \rho^+\in \mathbf{S}^1_{EK}(X^{A^-})$ and $\mathcal{H}^+\neq \emptyset$}. For each $h\in \mathcal{H}^+$, $(\EE^{-A,h},\FF^{-A,h})$ is a quasi-regular and irreducible symmetric Dirichlet form on $L^2(E,h^2\cdot m)$. 
\end{proposition}
\begin{proof}
Since $P^{-A}_th\leq h$, one can easily check that $(\EE^{-A,h},\FF^{-A,h})$ is a symmetric Dirichlet form on $L^2(E,h^2\cdot m)$. We only need to prove that it is quasi-regular and irreducible. The irreducibility can be deduced from
\[
	P_t^{-A,h}f(x)=\mathbf{E}_x\left( f(X_t)\mathrm{e}_{\mu+F}(t)\frac{h(X_t)}{h(X_0)}\right)
\]
and $\mathrm{e}_{\mu+F}(t)h(X_t)/h(X_0)>0$ for any $t<\zeta$. 

For the quasi-regularity, note that $\EE^{-A}_1\asymp \EE_1$ {by Lemma~\ref{LM21}, because $\EE^{-A}$ is non-negative definite under $\mathcal{H}^+\neq \emptyset$}. Thus the $\EE$-nest (resp. $\EE$-exceptional set, $\EE$-quasi-continuous function) is also an $\EE^{-A}$-nest (resp. $\EE^{-A}$-exceptional, $\EE^{-A}$-quasi-continuous, see \cite[Definition~4.1]{MR95}) and vice versa. Since $(\EE,\FF)$ is quasi-regular, it follows that $(\EE^{-A},\FF^{-A})$ satisfies \cite[Definition~4.9~(i, ii, iii)]{MR95}. Then the quasi-regularity of $(\EE^{-A,h},\FF^{-A,h})$ follows from \cite[Proposition~4.2]{MR95}, $h>0$ q.e. and $h$ is quasi-continuous. 
\end{proof}

The following proposition concludes that the subcriticality and criticality in Definition~\ref{DEF11} are well defined. 

\begin{proposition}\label{PRO24}
Assume that $\rho^-\in \mathbf{S}^1_D(X)$ and $\rho^+\in \mathbf{S}^1_{EK}(X^{A^-})$. If for some $h\in \mathcal{H}^+$, $(\EE^{-A,h}, \FF^{-A,h})$ is recurrent (resp. transient), then for any $h\in \mathcal{H}^+$, $(\EE^{-A,h}, \FF^{-A,h})$ is recurrent (resp. transient). 
\end{proposition} 
\begin{proof}
Take two functions $h, \tilde{h}\in \mathcal{H}^+$. Assume that $(\EE^{-A,h}, \FF^{-A,h})$ is transient but $(\EE^{-A,\tilde{h}}, \FF^{-A,\tilde{h}})$ is recurrent. By Proposition~\ref{PRO23} and \cite[Proposition~2.2]{G80}, there exists a strictly positive and bounded function $g$ such that $G^{-A,h}g$ is strictly positive and bounded, where $G^{-A, h}$ is the 0-order resolvent of $(\EE^{-A,h}, \FF^{-A,h})$. Clearly, $g\cdot h^2\cdot m$ is a positive smooth measure with respect to $(\EE^{-A,h}, \FF^{-A,h})$. Thus it follows from \eqref{EQ1FAH} and \eqref{EQ2EUN} that
\[
	\int_E u^2 g h^2dm \leq \|G^{-A,h}g\|_\infty \EE^{-A,h}(u,u)=\|G^{-A,h}g\|_\infty \EE^{-A}(uh,uh),\quad \forall u\in \FF^{-A,h}. 
\]
Then 
\begin{equation}\label{EQ2EUG}
\int_E u^2 gdm \leq \|G^{-A,h}g\|_\infty \EE^{-A}(u,u), \quad \forall u\in \FF^{-A}. 
\end{equation}
On the other hand, the recurrence of $(\EE^{-A,\tilde{h}}, \FF^{-A,\tilde{h}})$ implies there exists a sequence $\{u_n\}\subset \FF^{-A,\tilde{h}}$ such that $u_n\rightarrow 1$, $m$-a.e., and $\EE^{-A,\tilde{h}}(u_n,u_n)\rightarrow 0$ (Cf. \cite[Theorem~1.6.3]{FOT11}). By \eqref{EQ1FAH}, we have $v_n:= u_n\cdot \tilde{h}\in \FF^{-A}$, $v_n\rightarrow \tilde{h}$, $m$-a.e., and $\EE^{-A}(v_n,v_n)\rightarrow 0$. Then it follows from \eqref{EQ2EUG} that 
\[
\begin{aligned}
\int_E \tilde{h}^2 gdm&=\int_E \lim_{n\rightarrow \infty} v_n^2gdm \\
 &\leq  \liminf_{n\rightarrow \infty}\int_E v_n^2gdm  \\
 &\leq \liminf_{n\rightarrow\infty}  \|G^{-A,h}g\|_\infty\EE^{-A}(v_n,v_n)\\&=0.
\end{aligned}\]
Therefore, $\tilde{h}=0$, which conduces to the contradiction.
\end{proof}

\section{Existence of superharmonic functions}

In this section, we further assume $E$ is a locally compact separable metric space, $m$ is a Radon measure and $(\EE,\FF)$ is regular on $L^2(E,m)$. Since there always exists a quasi-homeomorphism between a quasi-regular Dirichlet form and another regular Dirichlet form by \cite{CMR94}, these further assumptions do not have essential effects except for the condition \textbf{(SF)}. 

The main result of this section is as follows.

\begin{theorem}\label{THM31}
Assume that $\rho^-\in \mathbf{K}(X)$. The subprocess of $X$ after the killing by $A^-$ is denoted by $X^{A^-}$. Further assume  $\mu^+\in \mathbf{K}_\infty(X^{A^-})$ and $F^+\in \mathbf{A}_\infty(X^{A^-})$.  Set
\[
	\lambda(\mu+F)=\inf\left\{\EE^{-A}(u,u): u\in \FF^{-A}, \ \int_E u^2d\rho^+=1 \right\},
\]
where $\rho^+$ is given by \eqref{EQ1RMN}. Then the Schr\"odinger form $(\EE^{-A},\FF^{-A})$ is
\begin{itemize}
\item[(1)] subcritical, if and only if $\lambda(\mu+F)>0$; 
\item[(2)] critical, if and only if $\lambda(\mu+F)=0$;
\item[(3)] supercritical, if and only if $\lambda(\mu+F)<0$. 
\end{itemize}
\end{theorem}

The proof {will be divided} into several parts. We first note that if $\lambda(\mu+F)<0$, then $\mathcal{H}^+=\emptyset$. In fact, if $\mathcal{H}^+\neq \emptyset$, then take a function $h\in \mathcal{H}^+$ and we have $\EE^{-A,h}(u,u)\geq 0$ for any $u\in \FF^{-A,h}$ by Proposition~\ref{PRO23}. From \eqref{EQ1FAH} we can deduce that $\EE^{-A}(u,u)\geq 0$ for any $u\in \FF^{-A}$. Thus $\lambda(\mu+F)\geq 0$. Therefore, $\lambda(\mu+F)<0$ implies the supercriticality of $(\EE^{-A},\FF^{-A})$. 

\subsection{Subcriticality}

The gaugeability was studied by many researchers such as \cite{C02, CS02, CS03, KK15, SS00, TU04}. It is known that the gaugeability has a very close connection with the subcriticality. The following proposition is an analogical result of \cite[\S5.1]{T14-2}. 

\begin{proposition}\label{PRO32}
Under the same assumptions as Theorem~\ref{THM31}, $\lambda(\mu+F)>0$ is equivalent to that the gauge function  $g_A(x):=\mathbf{E}^{A^-}_x \mathrm{e}_{\mu^++F^+}(\zeta)$ is bounded, where $\mathbf{E}^{A^-}_x$ is the expectation relative to $X^{A^-}$. Furthermore, if $\lambda(\mu+F)>0$, then $g_A\in \mathcal{H}^+$ and $(\EE^{-A,g_A},\FF^{-A,g_A})$ is transient, in other words, $(\EE^{-A},\FF^{-A})$ is subcritical. 
\end{proposition}
\begin{proof}
Clearly, $\rho^-\in \mathbf{S}^1_D(X)$ and $\rho^+\in \mathbf{S}^1_{EK}(X^{A^-})$. From the proof of Lemma~\ref{LM21}, we know that for any $u\in \FF$,
\[
	\EE^{-A}(u,u)=\EE^{A^-}(u,u)-\int u^2d\rho^+ +\int_{E\times E\setminus d}\left(u(x)-u(y)\right)^2G^+(x,y)J(dxdy).
\]
It follows from \cite[Theorem~1.1]{KK15} that $\lambda(\mu+F)>0$ is equivalent to that $g_A$ is bounded. Clearly, $g_A(x)\geq 1$. The quasi-continuity of $g_A$ follows from \cite[Theorem~2.6]{CS03}. The fact $P^{-A}g_A\leq g_A$ can be deduced through the same method as \cite[Lemma~5.2]{T14-2}. Therefore, $g_A\in \mathcal{H}^+$. 

Finally, we take to prove $(\EE^{-A,g_A},\FF^{-A,g_A})$ is transient. Denote
\[
	\EE^Y(u,u):=\EE^{-A}(u,u)+\int u^2d\rho^+. 
\]
Since $G^+$ is bounded, we may deduce that $\EE^{A-}\asymp \EE^Y$ and thus $(\EE^Y, \FF)$ is a transient Dirichlet form. By \cite[Proposition]{G80} and \eqref{EQ2EUN}, we may take a strictly positive, bounded and $m$-integrable function $g$ such that 
\[
	\int u^2gdm\leq \EE^Y(u,u),\quad \forall u\in \FF. 
\]
Since $\EE^Y(u,u)\geq (1+\lambda(\mu+F))\int u^2d\rho^+$, we have
\[
	\EE^{-A}(u,u)=\EE^Y(u,u)-\int u^2d\rho^+\geq \frac{\lambda(\mu+F)}{1+\lambda(\mu+F)}\EE^Y(u,u). 
\]
Hence for any $u\in \mathcal{F}^{-A,g_A}$, 
\[
\EE^{-A,g_A}(u,u)=\EE^{-A}(ug_A,ug_A)\geq \frac{\lambda(\mu+F)}{1+\lambda(\mu+F)} \int u^2g_A^2gdm\geq \left(\int |u|\tilde{g}d\tilde{m}\right)^2, 
\]
where $d\tilde{m}:=g_A^2dm$, $\tilde{g}:=\sqrt{\lambda}\cdot g/ \sqrt{(1+\lambda)\cdot \int gd\tilde{m}}$ and $\lambda:=\lambda(\mu+F)$. Clearly, $\tilde{g}$ is $\tilde{m}$-integrable. Therefore, $(\EE^{-A,g_A},\FF^{-A,g_A})$ is transient.
\end{proof}	
\begin{remark}
The assumptions in the above proposition could be much weaker.  
For example, when $\rho^-\in \mathbf{S}_D^1(X)$, $\mu^+\in \mathbf{K}_1(X^{A^-})$ and $F^+\in \mathbf{J}_\infty(X^{A^-})$, the above proof is still valid. 
\end{remark}

\subsection{Criticality}

Finally, we only need to prove the following proposition. 

\begin{proposition}\label{PRO34}
Under the same assumptions as Theorem~\ref{THM31}, $\lambda(\mu+F)=0$ implies the criticality of $(\EE^{-A},\FF^{-A})$. 
\end{proposition}

We first need the following lemma.

\begin{lemma}\label{LM35}
If $\rho^-\in \mathbf{K}(X)$, then $X^{A^-}$ satisfies \textbf{(SF)}.
\end{lemma}
\begin{proof}
Since the semigroup of $X^{A^-}$ is 
\[
	P^{A^-}_tf(x)=\mathbf{E}_x\left(f(X_t)M^{A^-}_t\right),\quad \forall f\in \mathcal{B}(E),
\]
with the multiplicative functional
\[
M^{A^-}_t:=\mathrm{e}^{-A^{\mu^-}_t}\prod_{s\leq t} \left(1-G^-(X_{s-},X_s)\right), 
\]
we only need to prove $\lim_{t\downarrow 0}\sup_{x\in E}\mathbf{E}_x\left| M^{A^-}_t-1\right|=0$ by \cite[Corollary~1.2]{CK09}. Indeed, let $G_-(x,y):=\mathrm{e}^{F^-(x,y)} -1$. Then $G_-(x,y)=G^-(x,y)\cdot \mathrm{e}^{F^-(x,y)}\leq \mathrm{e}^{\|F^-\|_\infty} G^-(x,y)$. It follows from $\rho^-\in \mathbf{K}(X)$ that $\mu^-+NG_-\cdot\mu_H\in \mathbf{K}(X)$. This implies
\[
\begin{aligned}
	\lim_{t\downarrow 0}\sup_{x\in E}\mathbf{E}_x&\left(A^{\mu^-}_t+\sum_{s\leq t}G_-(X_{s-},X_s) \right) \\&=\lim_{t\downarrow 0}\sup_{x\in E}\mathbf{E}_x\left(A^{\mu^-}_t+\int_0^t NG_-(X_s)dH_s \right)=0. 
\end{aligned}
\]
Denote the AF $A^{\mu^-+G_-}_t:= A^{\mu^-}_t+\sum_{s\leq t}G_-(X_{s-},X_s)$. By Khas'minskii's lemma (Cf. \cite[Lemma~2.1(a)]{Y97}), we have
\[
	\sup_{x\in E}\mathbf{E}_x \left(\mathrm{Exp}(A^{\mu^-+G_-})_t \right)\leq \frac{1}{1-\sup_{x\in E}\mathbf{E}_x A^{\mu^-+G_-}_t }
\]
for sufficiently small $t>0$. From this, we can deduce that
\[
\begin{aligned}
\lim_{t\downarrow 0}\sup_{x\in E}\mathbf{E}_x\left| M^{A^-}_t-1\right|&\leq \lim_{t\downarrow 0}\sup_{x\in E}\mathbf{E}_x\left| 1-\left(M^{A^-}_t\right)^{-1}\right| \\
&=\lim_{t\downarrow 0}\sup_{x\in E}\mathbf{E}_x\left| 1- \mathrm{Exp}(A^{\mu^-+G_-})_t\right| \\
&\leq \lim_{t\downarrow 0} \frac{\sup_{x\in E}\mathbf{E}_x A^{\mu^-+G_-}_t}{1-\sup_{x\in E}\mathbf{E}_x A^{\mu^-+G_-}_t }  \\
&=0.
\end{aligned}\]
That completes the proof.
\end{proof}

\begin{proof}[Proof of Proposition~\ref{PRO34}]
Note that $\EE^Y \asymp \EE^{A^-}$ and $\EE^{A^-}_1\asymp \EE_1$ by Lemma~\ref{LM21} and Proposition~\ref{PRO32}. Thus $(\EE^Y,\FF)$ is a regular Dirichlet form on $L^2(E,m)$. Denote its associated Hunt process by $Y$. Then $\lambda(\mu+F)=0$ is equivalent to 
\begin{equation}\label{EQ3LEY}
\lambda:= \inf \left\{\EE^{Y}(u,u): u\in \FF, \int_E u^2d\rho^+=1 \right\}=1.
\end{equation}
Thus we only need to prove $Y$ satisfies \textbf{(I)} and \textbf{(SF)}, and $\rho^+\in \mathbf{K}_\infty(Y)$. Then we have this proposition from \cite[Theorem~2.1]{T14} and \cite[\S5.2]{T14-2}.

The semigroup of $Y$ is
\[
\begin{aligned}
	P^Y_tf(x)&=\mathbf{E}_x\left(f(X_t)\mathrm{e}^{-A^-_t+\sum_{s\leq t}F^+(X_{s-},X_s)-\int_0^t NG^+(X_s)dH_s} \right) \\ &:=\mathbf{E}_x\left(f(X_t)M^Y_t\right). 
	\end{aligned}\]
Since $M^Y_t>0$ for $t<\zeta$, the irreducibility of $X$ implies the irreducibility of $Y$. On the other hand, $Y$ is transient by $\EE^Y\asymp \EE^{A^-}$ and the transience of $(\EE^{A^-},\FF^{A^-})$. Hence $Y$ satisfies \textbf{(I)}. 

Note that the L\'evy system of $X^{A^-}$ is $(N^-,H^-)=\left((1-G^-)\cdot N, H\right)$ {(Cf. \eqref{EQ3NXY})}, the Revuz measure of $H$ relative to $X^{A^-}$ is still $\mu_H$, and $1-G^-=\mathrm{e}^{-F^-}$ is bounded above and below since $F^-$ is bounded. The transform from $X^{A^-}$ to $Y$ may be decomposed into two parts: {the first one is killing by the measure $NG^+\cdot \mu_H$, and the second step is the perturbation induced by the AF $A^{F^+}$ (Cf. \eqref{EQ3LTA}).} 
We know that $X^{A^-}$ satisfies \textbf{(SF)} by Lemma~\ref{LM35}. Since $F^+\in \mathbf{A}_\infty(X^{A^-}) \subset \mathbf{J}(X^{A^-})$, it follows that $NG^+\cdot \mu_H\in \mathbf{K}(X^{A^-})$. {Applying Lemma~\ref{LM35} to $X^{A^-}$}, the subprocess $Z$ after the first step still satisfies \textbf{(SF)}. Particularly, the L\'evy system of $Z$ is also $(N^-,H)$. Similarly to the proof of Lemma~\ref{LM35}, one can also deduce that the perturbation in the second step remains the strong Feller property since $NG^+\cdot \mu_H\in \mathbf{K}(X^{A^-})\subset \mathbf{K}(Z)$. In other words, $Y$ satisfies \textbf{(SF)}. 

Finally, we shall prove $\rho^+\in \mathbf{K}_\infty(Y)$. Since $\mu^+\in \mathbf{K}_\infty(X^{A^-})$ and $F^+\in \mathbf{A}_\infty(X^{A^-})\subset \mathbf{J}_\infty(X^{A^-})$, it follows that $\rho^+\in \mathbf{K}_\infty(X^{A^-})$. 
On the other hand, from \eqref{EQ3LEY} we can deduce that 
\[
	\EE^Y(u,u)\geq \int_Eu^2d\rho^+ \geq \int_E u^2d(NG^+\cdot \mu_H),\quad \forall u\in \FF.
\]
This implies
\[
\inf\left\{\EE^Y(u,u): u\in \FF, \int_E u^2d(NG^+\cdot \mu_H)=1\right\}\geq 1>0.
\]
Thus $(X^{A^-}, -NG^+\cdot \mu_H+ F^+)$ is gaugeable (i.e. $\sup_{x\in E}\mathbf{E}^{A^-}_x \mathrm{e}_{-NG^+\cdot \mu_H+F^+}(\zeta)<\infty$) by \cite[Theorem~1.1]{KK15}. Denote the $0$-order resolvents of $X^{A^-}$ and $Y$ by $G^{A^-}(x,y)$ and $G^Y(x,y)$. It follows from \cite[Lemma~3.9~(1) and Theorem~3.10]{C02} that there exists a constant $K>0$ such that 
\[
	G^Y(x,y)\leq  KG^{A^-}(x,y)
\]
 for any $x,y$.  Therefore, $\mathbf{K}_\infty(X^{A^-})\subset \mathbf{K}_\infty(Y)$ and $\rho^+\in \mathbf{K}_\infty(Y)$. That completes the proof.
\end{proof}

\begin{remark}
The condition $F^+\in \mathbf{A}_\infty(X^{A^-})$  in Theorem~\ref{THM31} is only used in the above proof to guarantee the conditional gauge theorem for $(X^{A^-}, -NG^+\cdot \mu_H+F^+)$ (Cf. \cite[Theorem~3.8]{C02}). 
\end{remark}

\section*{Acknowledgement}

Part of this work was finished when the author visited Professor Masayoshi Takeda at Tohuku University in February of 2016. Many thanks to his hospitality and helpful discussions. {The author is grateful to Professor Jiangang Ying for advising him to consider this problem.} Very interesting discussions with Professor Zhi-Ming Ma are also gratefully acknowledged.

\bibliographystyle{amsplain}




\end{document}